\documentclass{amsart}
\usepackage{amssymb,latexsym}
\theoremstyle{plain}
\newtheorem{theorem}{Theorem}

\newtheorem{proposition}{Proposition}
\newtheorem{lemma}{Lemma}
\theoremstyle{definition}
\newtheorem{definition}{Definition}
\newtheorem{example}{Example}
\newtheorem{notation}{Notation}
\newtheorem{remark}{Remark}

\newcommand{\enm}[1]{\ensuremath{#1}} 

\newcommand{\PP}{\enm{\mathbb {P}}}
\newcommand{\NN}{\enm{\mathbb {N}}}

\newcommand{\CC}{\enm{\mathbb {C}}}
\newcommand{\Oo}{\enm{\mathcal {O}}}

\newcommand{\Uu}{\enm{\mathcal {U}}}

\newcommand{\Rr}{\enm{\mathcal {R}}}
\newcommand{\Cc}{\enm{\mathcal {C}}}

\newcommand{\RR}{\enm{\mathbb {R}}}
\newcommand{\Ss}{\enm{\mathcal {S}}}
\newcommand{\Vv}{\enm{\mathcal {V}}}

\begin{document}

\title[Typical and Admissible ranks over fields]
{Typical and Admissible ranks over fields}
\author{Edoardo Ballico}
\address{Dept. of Mathematics\\
 University of Trento\\
38123 Povo (TN), Italy}
\email{ballico@science.unitn.it}
\author{Alessandra Bernardi}
\address{Dept. of Mathematics\\
 University of Trento\\
38123 Povo (TN), Italy}
\email{alessandra.bernardi@unitn.it}
\thanks{The author was partially supported by MIUR and GNSAGA of INdAM (Italy).}

\begin{abstract} Let $X(\RR)$ be a  geometrically connected variety defined over $\RR$  
and such that the set of all its (also complex) points $X(\CC)$ is non-degenerate.
 We introduce the notion of \emph{admissible rank}   
 of a point $P$ with respect to $X$  to be  the minimal cardinality of a set of points of $X(\CC)$ such that $P\in \langle S \rangle$ that is stable under conjugation. 
 Any set evincing the admissible rank can be equipped with a \emph{label} keeping track of the number of its complex and real points. We show that in the case of generic identifiability there is an open dense euclidean subset of points with certain admissible rank for any possible label. Moreover we show that if $X$ is a rational normal curve than there always exists a label for the generic element. We present two examples in which either the label doesn't exists or the admissible rank is strictly bigger than the usual complex rank.
\end{abstract}

\maketitle
\section*{Introduction}
A very important problem in the framework of Tensor Decomposition is to understand when a given real tensor $T$ can be written as a linear combination of real rank 1 tensors with a minimal possible number of terms, and that number is called the rank of $T$. For the applications it is also often very useful to know which are the typical ranks, i.e. the ranks for which it is possible to find an euclidean open set  of tensors with that given rank. For the specific case of tensors, computing the rank of $T$  corresponds to find the smallest  projective space $s$-secant to the real Segre variety containing $[T]$. One can clearly define all these concepts for any variety $X$ and saying that the real rank of $P\in \langle X\rangle$ is the minimal cardinality of a set of points of $X$ whose span contains $P$.
In order to be as much as general as possible we will always work with a real field $K$ instead of over $\mathbb{R}$ and its real closure  $\mathcal{R}$. We will indicate with $\Cc := \Rr (i)$ the algebraic closure of $\Rr$, but the reader not interested in abstract fields may take $\RR$ instead of $\Rr$ and
$\CC$ instead of $\Cc$.

\def\niente{
A first possible generalization is to  work over a real field $K$ instead of over $\mathbb{R}$, and to define the rank with respect to a geometrically integral non degenerate projective variety $X$ defined over $K$ (instead of restricting ourselves to the specific  case of Segre variety) such
that $X(K)$ is Zariski dense in $\overline{X}:=X(\overline{K})$ and such that $X$ is embedded in a projective space. This is what we do in Section \ref{first}. 
For each  $P\in \PP^r(K)$ the  \emph{$X(K)$-rank} $r_{X(K)}({P})$ of $P$ is the minimal cardinality of a subset  $S\subset X(K)$
such that $P\in \langle S\rangle$, where $\langle \  \rangle$ denotes the linear span. 
We will say that if 
\begin{equation}\label{UKa}U_{K,a}:= \left\{P\in \PP^r(K)\; | \; r_{X(K)}(P)=a\right\}\end{equation}
has a non empty interior with respect to the Zariski topology, then $a$ is a \emph{$K$-dense rank} (see Definition \ref{dense}).
If $\mathcal{R}$ is the real closure of $K$, the definition of \emph{$\mathcal{R}$-typical rank} is the same of ``typical rank" over $\mathbb{R}$ (see also Definition \ref{typical}).
In Theorem \ref{b2} we prove that every $\mathcal{R}$-typical rank is also a $K$-dense rank.

\smallskip

A  further generalization could be to define the rank with respect to the join of various varieties $X_i(K)\subset \mathbb{P}^r(K)$, instead of one single variety. This is what we do in Section \ref{second}.
More precisely the \emph{rank $r_{X_i,i\ge 1}(P)$} of a point $P\in \mathbb{P}^r(K)$ is the minimal cardinality of a finite set $I\subset \NN\setminus \{0\}$ such that there is $P_i\in X_i(K)\subset \mathbb{P}^r(K)$, $i\in I$,
with $P\in \langle \{P_i\}_{i\in I}\rangle$ (see Definition \ref{rank:join}). 
Clearly if all the $X_i(K)$'s are the same variety, then the previous definition coincides with this one.

In this setting, the space $U_{K,a}$ of (\ref{UKa}) can be generalized to be the set of all points of $\PP^r(K)$ with
rank $a$ with respect to the sequence $X_i$, $i\ge 1$.

In this more general setting, the definitions of $K$-dense and typical rank, the  statements and the proof of Theorem \ref{b2} work verbatim, except that  we
require that each $X_i(K)$ is dense in $X(\Rr)$ in the euclidean topology.

The more general set-up covers the following cases.
\begin{example}
The space $\PP^r$ is the Segre embedding of the multiprojective space $\PP^{n_1}\times \cdots \times \PP^{n_k}$, $r+1= \Pi_{i=1}^k(n_i+1)$, and each $X_i$ is a closed subvariety of $\PP^{n_1}\times \cdots \times \PP^{n_k}$. For instance, each $X_i$ may be a smaller multiprojective space (depending
on less multihomogeneous variables). If $X_i =\PP^{n_1}\times \cdots \times \PP^{n_k}$ for sufficiently many indices, any rank is achieved by a set $I$ such that $X_i =\PP^{n_1}\times \cdots \times \PP^{n_k}$ for
all $i$, but even in this case there may be cheaper sets $J$ (i.e. with $\sharp (J) =\sharp (I)$,  but $X_i\subsetneq \PP^{n_1}\times \cdots \times \PP^{n_k}$ for some $i\in J$).
\end{example}

\begin{example}\label{ooo1}
Instead of the Segre embedding of the multiprojective space $\PP^{n_1}\times \cdots \times \PP^{n_k}$ we may take the Segre-Veronese embedding of multidegree
$(d_1,\dots ,d_k)$ of $\PP^{n_1}\times \cdots \PP^{n_k}$; here $r+1 =\prod _{i=1}^{k} \binom{n_i+d_i}{n_i}$. \end{example}

\smallskip
}

We introduce the notion of \emph{admissible rank} $r_{X,\Rr}(P)$  of a point $P$ (Definition \ref{admissible}) to be  the minimal cardinality of a set $S\subset X(\Cc)$ that is stable under the conjugation action and
such that $P\in \langle S\rangle$. 
Any such a set can be labelled by keeping track of the number of the real points. 
Clearly if $S$ evinces the rank of $P$ than the label of $S$ will be also \emph{a label for $P$}.

One can ask various questions regarding those notions of admissible rank and label of a decomposition of a point.

First of all when a point has a unique decomposition, then clearly the label is uniquely associated to it. In this case it is even possible to prove  that if the generic element is identifiable, then the set of all real points with that rank is dense for the euclidean topology  (see Theorem \ref{i0}).

A ``~partial particular case~" of this situation is the rational normal curve. When the curve has odd degree, then the generic homogeneous bivariate polynomial of that degree is identifiable, so in this case we can apply the result just described. But if the rational normal curve is of even degree, then Theorem \ref{i0} does not apply. Anyway we can completely describe the situation for rational normal curves with the following theorem.

\begin{theorem}\label{oo1.1}
Let $X\subset \PP^d$, $d\ge 2$, be a rational normal curve defined over $\Rr$. Then there is a non-empty open subset $\Uu \subset \PP^d(\Rr )$ such that
$\PP^d(\Rr )\setminus \Uu$ has euclidean dimension $< d$ and each $q\in \Uu$ has admissible rank $\lceil (d+1)/2\rceil$.
\end{theorem}

The behaviur of rational normal curves is peculiar. In fact it is not always true that when we don't have generic identifiability  one can  find a label for the generic element (remark that the identifiability of the generic element is quite rare).  In  Secion \ref{ex2} we show an example by using an elliptic curve: we can explicitly build an euclidean neighborhood of a generic point in $\PP^3(\Rr)$ where no points in that neighborhood have a label with respect to the Variety of Sum of Powers $VSP(P)$ (see Definition \ref{vsp}). However in this case $VSP(P)$   is finite
and we do not have positive dimensional examples that could be very interesting if they exist. 

The last question that we like to address is what happens when the complex rank is smaller than the admissible rank. For a given $P\in \PP^r(\Rr )$ consider the set $\Ss (P,X,\Cc )$ of all $S\subset X(\Cc )$ evincing $r_{X(\Cc )}({P})$. First of all it is a constructible
subset of $\PP^r(\Cc )$ invariant by the conjugation action, hence it is defined over $\Rr$ and $S\in \Ss (P,X,\Cc )(\Rr )$ if and only if $S$ is fixed by the conjugation action, i.e. if and only if $S$ has a label $(s,a)$
for some $a$ with $0\le a\le \left\lfloor \frac{r_{X(\Cc )}({P})}{2}\right\rfloor$. Therefore $P$ has admissible rank $r_{X(\Cc )}({P})$ if and only $\Ss (P,X,\Cc )(\Rr ) \ne \emptyset$, not even for $P$ outside a subset of $\PP^r(\Rr )$ with
euclidean dimension $<r$. Therefore if $r_{X(\Cc )}({P})<r_{X,\Rr}(P)$ then $\Ss (P,X,\Cc )(\Rr ) = \emptyset$. In Example \ref{iii1} we show this behavior by constructing a very special homogeneous polynomial  in $n\ge 2$ variables and even degree $d\ge 6$  such that $r_{X(\Cc )}({P})= 3d/2<r_{X,\Rr}({P})$, $\sharp (\Ss (P,X,\Cc )) =2$ and $\Ss (P,X,\Cc )(\Rr )=\emptyset$.

\bigskip

Acknowledgments: We want to thank Giorgio Ottaviani for very helpful and constructive remarks.

\section{$K$-dense and $K$-typical ranks}\label{first}
 Let $K$ be any field with characteristic zero and let $\overline{K}$ denote its algebraic closure. Let $X$ be a geometrically integral projective variety defined over $K$ such
that $X(K)$ is Zariski dense in $\overline{X}:=X(\overline{K})$. We fix an inclusion $X\subset \PP^r$ defined over $K$ and such that $X$ spans $\PP^r$, i.e.
no hyperplane defined over $\overline{K}$ contains $X(\overline{K})$. Since $X(K)$ is assumed to be Zariski dense in $\overline{X}$, the non-degeneracy
of $X$ is equivalent to assuming that $X(K)$ spans $\PP^r(K)$ over $K$. 
\\
Now, $\PP^r(\overline{K})$ and $\PP^r(K)$ have their own Zariski topology and the Zariski topology of $\PP^r(K)$ is the restriction of the one on $\PP^r(\overline{K})$. 
\begin{definition}[$\overline{X}$-rank and $X(K)$-rank]
For
each point $P\in \PP^r(\overline{K})$ (resp. $P\in \PP^r(K)$) the \emph{$\overline{X}$-rank} $r_{\overline{X}}({P})$ (resp. the \emph{$X(K)$-rank} $r_{X(K)}({P})$) of $P$ is the minimal cardinality of a subset $S\subset \overline{X}$ (resp. $S\subset X(K)$)
such that $P\in \langle S\rangle$, where $\langle \  \rangle$ denote the linear span. We say that $S$ \emph{evinces} $r_{\overline{X}}({P})$.
\end{definition}
\begin{definition}\label{dense}
For each integer $a>0$ let
\begin{equation}\label{UKa}U_{K,a}:= \left\{P\in \PP^r(K)\; | \; r_{X(K)}(P)=a\right\}\end{equation}
be the subset in $\PP^r(K)$ of the points of fixed $\overline{X}$-rank equal to $a$ as in (\ref{UKa}).
We say
that $a$ is a \emph{$K$-dense rank} if the set $U_{K,a}$ is Zariski dense in $\PP^r(\overline{K})$.
\end{definition}
 Since $\PP^r(K)$ is Zariski dense in $\PP^r(\overline{K})$, then $U_{K,a}$ is a dense
subset of $\PP^r(K)$. 
\begin{definition}[Generic $X(\overline{K})$-rank]  The \emph{generic $X(\overline{K})$-rank} of $\PP^r(\overline{K})$ is the  $\overline{X}$-rank of the generic element of $\PP^r(\overline{K})$.
\end{definition}
The minimal $K$-dense rank is just the generic $X(\overline{K})$-rank of $\PP^r(\overline{K})$ (Lemma \ref{c1}).
\begin{definition}\label{real} A field $K$ is \emph{real} if  $x_1, \ldots , x_n\in K$  are such that $\sum_{i=1}^n x_i^2=0$ then $x_i=0$ for all $i=1, \ldots , n$.
\end{definition}
We recall that a field  $K$ admits an ordering if and only if $-1$ is not a sum of squares (\cite[Theorem 1.1.8]{bcr}). It is possible to prove (cfr \cite[Chapter 4]{bcr}) that when $K$ admits an ordering, then the filed $K$ is real.

\smallskip 

In the real field $\mathbb{R}$ the notion of \emph{typical rank} has been introduced by various authors (see e.g. \cite{bbo}, \cite{b}, \cite{bs}, \cite{bt}, \cite{co}): An integer $r$ is said to be a \emph{$X$-typical rank} if the set $U_{\mathbb{R},r}$ contains a non-empty subset for the euclidean topology of $\PP^r(\mathbb{R})$.

We can also introduce this notion of typical rank into our setting, but we have to assume that $K$ is real closed.

\begin{definition}\label{closed}
A real field is \emph{closed} if it does not have trivial real algebraic extensions.
\end{definition}
We recall that for any ordering $\le $ of $K$, there is a unique inclusion $K\to \Rr$
of ordered field with $\Rr$ real closed (\cite[Theorem 1.3.2]{bcr}).

 If $K$ is real closed, then the sets $X(K)$ and $\PP^r(K)$ have the euclidean topology in the sense of \cite[page 26]{bcr} where the euclidean topology of $K^n$ comes from the ordering structure of $K$, i.e. the \emph{euclidian topology} on $K^n$ is the topology for which open balls form a basis of open subsets (cft. \cite[Definition 2.19]{bcr}).
 
\begin{definition}[Typical $X(K)$-rank]\label{typical} We say that
$a$ is a \emph{typical $X(K)$-rank} if $U_{K,a}$ contains a non-empty subset for the euclidean topology of $\PP^r(K)$.
\end{definition} 
As in the case $K =\mathbb {R}$ the minimal typical rank
is the generic rank (\cite[Theorem 2]{bt}). The set of all typical ranks has no gaps, i.e. if $a$ and $b\ge a+2$ are typical ranks, then $c$ is typical if $a<c<b$ (\cite[Theorem 2.2]{bbo}). 

\begin{notation}
If $K$ is contained in a field $F$ and $S\subseteq \PP^r(K)$, let $\langle S\rangle _F$ denote the linear span of $S$ in $\PP^r(F)$.
\end{notation}

\begin{lemma}\label{c1}
Each set $U_{\overline{K},a}$ is constructible. If $K$ is real closed, then each set $U_{K,a}$ is semialgebraic.
\end{lemma}

\begin{proof}
Since $U_{\overline{K},1} =X(\overline{K})$ is Zariski closed, we may assume $a>1$ and that $G:= \cup _{c<a} U_{\overline{K},c}$ is constructible.
\\ 
There is an obvious morphism from  the set  $E\subset X(\overline{K})^a$ of all $a$-uple of linearly independent
points 
to the Grassmannian $G(a-1,r)(\overline{K})$ of all
$(a-1)$-dimensional $\overline{K}$-linear subspaces of $\PP^r (\overline{K})$:
$$\phi:E \to G(a-1,r)(\overline{K}).$$
As usual let $I:= \{(x,N)\in \PP^r(\overline{K})\times G(a-1,r)(\overline{K})\, | \,  x\in N\}$ be the incidence
correspondence and let $\pi _1: I\to \PP^r(\overline{K})$
and $\pi _2: I\to  G(a-1,r)(\overline{K})$ denote the restrictions to $I$ of the two projections.\\
Now  $U(\overline{K},a)$ is the intersection with $\PP^r(\overline{K})\setminus G$ of $\pi _1(\pi _2^{-1}(\phi (E))$. Obviously
the counterimage by a continuous map for the Zariski topology of a constructible set is constructible.
Use that the image of a constructible set is constructible by a theorem of Chevalley \cite[Exercise II.3.19]{h}. If $K$ is real closed, it is sufficient to quote \cite[Proposition 2.2.3]{bcr}, instead of Chevalley's theorem.
\end{proof}

\begin{proposition}\label{b1}
If $K$ is a real closed field, then an integer is $K$-typical if and only if it is $K$-dense.
\end{proposition}

\begin{proof}
Since $U_{K,a}$ is semialgebraic (Lemma \ref{c1}), it
is Zariski dense in $\PP^r(K)$ if and only if it has dimension $r$ (or, equivalently, if and only if it contains a non-empty open subset in the euclidean topology). Since $K$ is infinite, $\PP^r(K)$
is Zariski dense in $\PP^r(\overline{K})$. Hence
$U_{K,a}$ is Zariski dense in $\PP^r(\overline{K})$ if and only if it is Zariski dense in $\PP^r(K)$. Thus $a$ is $K$-dense if and only if it is $K$-typical.
\end{proof}

\begin{remark}\label{b2.1}
Assume $K$ real closed and let $L\supset K$ be any real closed field containing $K$. By Proposition \ref{b1} the $K$-typical ranks of $X(K)\subset \PP^r(K)$ and of $X(L)\subset \PP^r(L)$
are the same (this also may be proved directly from the Tarski-Seidenberg principle). In particular this means that  the typical ranks of real tensors and the typical rank of real homogeneous polynomials are realized over the real closure of $\mathbb {Q}$. Moreover, if $a$ is typical, then $U_{K,a}$ is dense in the interior of $U(L,a)$ for the euclidean topology. 
\end{remark}
A very important result (see \cite[Chapter 4]{bcr}) is that for any real field $K$, there exists a unique real closed field $\Rr$ such that $K\subset \Rr$.

\begin{theorem}\label{b2}
Assume that $K$ admits an ordering, $\le $, and let $\Rr$ be the real closure of the pair $(K,\le )$. Assume that $X(K)$ is dense in $X(\Rr )$ in the euclidean topology. Then every $\Rr$-typical rank of $X(\Rr)$ is a $K$-dense rank for $X(K)$.
\end{theorem}

\begin{proof}
Let $U\subset \PP^r(\Rr)$ be an open subset for the euclidean topology formed by points
with fixed  $X(\Rr )$-rank equal to $a$. Let $E\subset X(\Rr )^a$ be the set of all linearly independent $a$-ples
$(Q_1,\dots ,Q_a)$ of distinct points. For each $(Q_1,\dots ,Q_a)\in E$, we have an $(a-1)$-dimensional $\Rr$-linear
space $\langle \{Q_1,\dots ,Q_a\}\rangle _{\Rr}$. By assumption there is an open subset $F\subset E$ in the euclidean topology
such that the union $\Gamma$ of all $\langle \{Q_1,\dots ,Q_a\}\rangle _{\Rr}$ contains all points of $U$. Since $X(K)$ is dense in $X(\Rr )$ in the euclidean topology, the set $E(K)$ is
dense in $E$ in the euclidean topology. Hence the subset $\Gamma '$ of $\Gamma$ formed by the $\Rr$-linear spans of elements of $E(K)$
is dense in $\Gamma$ and hence its closure in the euclidean topology contains $U$. Since the closure in the euclidean topology of the set $\Gamma '' \subset \Gamma'$ formed
by the $K$-linear spans of elements of $E(K)$ contains $\Gamma '$, every $\Rr$-typical rank is a $K$-dense rank.
\end{proof}

\begin{remark}
The condition that $X(K)$ is Zariski dense in $X(\overline{K})$ is very restrictive if $K$ is a number field and $X$ has general type. For example if $X$ is a curve of geometric genus  $\ge 2$, then it is never satisfied.
But it is not restrictive in the two more important cases: tensors (where $X$ is a product of projective spaces
with the Segre embedding)
and degree $d$-homogeneous polynomials (where $X =\PP ^n$, and the inclusion $X\subset \PP^{\binom{n+d}{n}-1}$ is the order $d$ Veronese embedding).
It applies also to Segre-Veronese embeddings of multiprojective spaces (the so-called Segre-Veronese varieties). 
\\
If $K$ has an ordering $\le $
and $\Rr$ is the real closure of $(K,\le )$, then $X(K)$ is dense in $X(\Rr )$ for the euclidean topology, because $K^n$ is dense in $\Rr^n$ for the euclidean topology.
The set of typical $X$-ranks may be very large (\cite{b}, \cite[Theorem 1.7]{bs} and \cite[Theorem 2.2]{bbo}).
\end{remark}

\section{Join, set-theoretic $K$-join and $K$-join}\label{second}

We describe now the situation of more than only one variety. First of all we need to distinguish if the join of two ore more varieties is defined over $\PP^r(\overline{K})$ or over $\PP^r(K)$. This will allow to introduce the notion of  rank with respect to join varieties and the ``~label~"  associated to a decomposition of an element. We will label a decomposition of a point with the number of its $\mathbb{P}^r(K)$ elements.
\begin{definition}[Join] Let $X, Y\subset \PP^r(\overline{K})$ be integral varieties over $\overline{K}$. 
We define the \emph{join $[X;Y]$} of $X$ and $Y$ to be the closure in $\PP^r(\overline{K})$ of the union of all lines spanned by a point of $X(\overline{K})$
and a different point of $Y(\overline{K})$ (if $X$ and $Y$ are the same point $Q$ we set  $[X;Y]:= \{Q\}$).
\end{definition}
The set $[X;Y]$ is an integral variety of dimension at most $\min \{r,\dim (X)+\dim (Y)+1\}$.
\begin{definition}
 If we have $s \ge 3$ integral varieties $X_i\subset \PP^r(\overline{K})$, $1\le i\le s$, we define inductively their join $[X_1;\cdots ;X_s]$ by the formula 
$$[X_1;\dots ;X_s]: = [[X_1;\dots ;X_{s-1}];X_s].$$ 
\end{definition}
The join is symmetric in the $X_i$'s. If $X$ and $Y$ are defined over $K$, then $[X;Y]$ is defined over $K$ and $[X;Y](K)$ contains the closure (in the Zariski topology) of $[X;Y]\cap \PP^r(K)$,
but it is usually larger (even if $K$ is real closed and $X(K)$ and $Y(K)$ are dense in $X(\overline{K})$ and $Y(\overline{K}))$.
\begin{definition}[Set theoretic $K$-join and $K$-join]
Assume that $X_i(K)$ is Zariski dense in $X_i(\overline{K})$ for all $i$. The \emph{set-theoretic $K$-join} $((X_1;\dots ;X_s))_K \subseteq \PP^r(K)$ of $X_1,\dots ,X_s$ is the union
of $K$-linear subspaces spanned by points $Q_1,\dots ,Q_s$ with $Q_i\in X_i(K)$ for all $i$.

The \emph{$K$-join} $[X_1;\dots ;X_s]_K$ is the Zariski closure in $\PP^r(K)$ of $((X_1;\dots ;X_s))_K$.
\end{definition}

Take geometrically integral projective varieties $X_i$, $i\ge 1$, defined over $K$ and equipped with an embedding $X_i\subset \PP^r$ defined over $K$. We allow the case $X_i=X_j$ for some $i\ne j$ (in the case $X_i=X$ for all $i$'s we would just get the set-up of Proposition \ref{b1} and Theorem \ref{b2}).
We assume that each $X_i(K)$ is Zariski dense in $X_i(\overline{K})$ and that $((X_1;\dots ;X_h))_K= \PP^r(K)$ for some $h$.
The latter condition implies that $[X_1;\dots ;X_h] =\PP^r(\overline{K})$.

\begin{definition}[Rank and label with respect to join]\label{rank:join}
Fix a point $Q\in \PP^r(K)$. The \emph{rank $r_{X_i,i\ge 1}(Q)$} is the minimal cardinality of a finite set $I\subset \NN\setminus \{0\}$ such that there is $Q_i\in X_i(K)\subset \mathbb{P}^r(K)$, $i\in I$,
with $Q\in \langle \{Q_i\}_{i\in I}\rangle$. Any $I$ as above will be called a {\emph{label}} of $Q$.
\end{definition}
In this new setup we can re-define $U_{K,a}$ of (\ref{UKa}) more generally.
\begin{notation} Let $U_{K,a}$ denote the set of all points of $\PP^r(K)$ with fixed 
rank $a$ with respect to the sequence of varieties  $X_i$, $i\ge 1$ as in Definition \ref{rank:join}.
\end{notation}
With this extended notion of $U_{K,a}$ the Definitions \ref{dense} and \ref{typical} of $K$-dense and typical rank respectively can be re-stated here verbatim.

\begin{remark}\label{b3}
The statements and proof of Proposition \ref{b1} and Theorem \ref{b2} work verbatim in this more general setting (for the generalization of Theorem \ref{b2} we
require that each $X_i(K)$ is dense in $X(\Rr)$ in the euclidean topology).
\end{remark}

The more general set-up clearly covers the tensor cases where $X$ is the Segre or the Segre-Veronese variety. 
For example it applies to any $X_i$ being closed subvarieties of $\PP^{n_1}\times \cdots \times \PP^{n_k}$ and then considering the Segre embedding of the multiprojective space $\PP^{n_1}\times \cdots \times \PP^{n_k}$ into $\mathbb{P}^r$ with $r+1= \Pi_{i=1}^k(n_i+1)$.
For instance, each $X_i$ may be a smaller multiprojective space (depending
on less multihomogeneous variables). If $X_i =\PP^{n_1}\times \cdots \times \PP^{n_k}$ for sufficiently many indices, any rank is achieved by a set $I$ such that $X_i =\PP^{n_1}\times \cdots \times \PP^{n_k}$ for
all $i$, but even in this case there may be cheaper sets $J$ (i.e. with $\sharp (J) =\sharp (I)$,  but $X_i\subsetneq \PP^{n_1}\times \cdots \times \PP^{n_k}$ for some $i\in J$).
\\
The case of Segre-Veronese embedding of multidegree
$(d_1,\dots ,d_k)$ of $\PP^{n_1}\times \cdots \PP^{n_k}$ (here $r+1 =\prod _{i=1}^{k} \binom{n_i+d_i}{n_i}$), is completely analogous.

\section{On the ``~generically identifiable~" case}\label{third}

Let $\Rr$ be a real closed field as in Definitions \ref{real} and \ref{closed} and take $\Cc := \Rr (i)$ to be the algebraic closure of $\Rr$ (for these fundamental facts we always refer to \cite{bcr}). Now $X$ is a geometrically connected variety defined over $\Rr$ with a fiexd embedding $X\subset \PP^r$ and we assume that $X(\Cc)$ is non-degenerate,
i.e. $X(\Cc)$ spans $\PP^r(\Cc)$.

\begin{definition}[$X$-rank]\label{label1}
 For each $P\in \PP^r(\Cc)$ the \emph{$X$-rank }of $P$
is a minimal cardinality of a set $S\subset X(\Cc)$ such that $P\in \langle S\rangle$. 
\end{definition}

We want to consider now the ``~conjugation action~" on the elements appearing in a decomposition of $P$.

\begin{notation}\label{not} Let $\sigma : \Cc \to \Cc$ be the field automorphism with $\sigma (x) =x$ for all $x\in \Rr$ and $\sigma (i) =-i$.

\end{notation}
Remark that the map $\sigma$ just introduced acts (not algebraically) on $X(\Cc )$ and $\PP^r(\CC )$
with $\sigma ^2$ as the identity and with $X(\Rr )$ and $\PP^r(\Rr)$ as its fixed point set. 
Note that if $\dim (X)=n$ and $\Rr = \RR$, then $\dim (X(\CC )\setminus X(\RR )) =2n$ and hence, topologically, a pair of complex conjugate
points of $X(\CC)\setminus X(\RR)$ ``~costs~" as two points of $X(\RR)$. It seems to us that the same should be true for
the algorithms used in the case of the tensor decomposition or the decomposition of degree $d$ homogeneous
polynomial as a sum (with signs if $d$ is even) of $d$ powers of linear forms.

\begin{definition}[Admissible rank]\label{admissible}
The {\emph{admissible rank}} $r_{X,\Rr}(P)$ of a real point $P\in \PP^r(\Rr )$ is the minimal cardinality of a possibly complex set $S\subset X(\Cc)$ such that $P\in \langle S\rangle$ and $\sigma (S) =S$. 
\end{definition}

As we did in Definition \ref{rank:join}, we can again keep track of the elements appearing in a decomposition of a point by labelling those that are real.
\begin{definition}[Label of $S$]
In any finite set $S\subset X(\Cc)$ such that $\sigma (S) =S$ there are $a$ pairs of $\sigma$-conjugated points of $X(\Cc )\setminus X(\Rr)$, with $0\le a \le \left\lfloor \frac{\sharp(S)}{2}\right\rfloor$ (the other  $(\sharp(S)-2a)$ points of $S$ are in $X(\Rr)$). 
 
 We say that $(\sharp(S),a)$
is {\emph{the label}} of $S$. If  $S$ evinces the admissible rank of $P$, then we say that $(r_{X,\Rr}({P}),a)$ is {\emph{a label}} of $P$.
\end{definition}

\begin{remark}\label{stup1}
Fix a finite set $S\subset X(\Cc)$. The set $\sigma (S)$ is finite and $\sharp (\sigma (S)) =\sharp (S)$. It is very natural to say that $S$ is defined over $\Rr$ if and only if
$\sigma (S) =S$. Clearly if  $\sigma (S)=S$,  $\sigma$ sends bijectively $S$ onto itself and the label $a$ of Definition \ref{stup1} is the number
of pairs $\{p,\sigma ({p})\}$, $p\in X(\Cc )\setminus X(\Rr )$, contained in $S$, while $\sharp(S)-2a = \sharp (X(\Rr )\cap S)$.
\end{remark}

If $P\in \PP^r(\Rr)$ has a unique decomposition, then we can directly associate a unique label to $P$ itself.

\begin{notation}[Label of $P$]
Fix $P\in \PP^r(\Rr)$ such that there is a unique set $S\subset X(\Cc )$ evincing the $X$-rank, $s$, of $P$. Since $\sigma (P)=P$, the uniqueness
of $S$ implies $\sigma (S) =S$. Hence $S$ has a label $(s,a)$ and we say that \emph{$P$ has label $(s,a)$}. If $X(\Rr )=\emptyset$, then each label is of the form $(2a,a)$. 
\end{notation}
\begin{definition}\label{secant}
For every integer $t\ge 1$ the \emph{$t$-secant variety} $Sec _t(X(\Cc ))$ of $X(\Cc)$
is the closure in $\PP^r(\Cc)$ of the set of all points with $X$-rank $t$.
\end{definition}
The set $Sec _t(X(\Cc))$ is an integral variety defined over $\Rr$ and we
are interested in its real locus $\sigma _t(X(\Cc))\cap \PP^r(\Rr)$. 
\begin{definition}\label{identifiable}
We say that $Sec _t(X(\Cc))$ is \emph{generically identifiable}
if for a general $P\in Sec _t(X(\Cc))$ there is a unique $S\subset X(\Cc )$ such that $\sharp (S)=t$ and $P\in \langle S\rangle$ (this notion has been already widely introduced in literature, see e.g. \cite{bco}, \cite{chio}, \cite{cov}).
\end{definition}
We have $\Cc = \Rr +\Rr i$ and hence we may see $\PP^r(\Cc )$ as an $\Rr$-algebraic variety of dimension $2r$. Hence $\PP^r(\Cc )$ has an euclidean topology
and this topology is inherited by all subsets of $\PP^r(\Cc)$. This topology on $Sec _t(X(\Cc ))$ is just the euclidean topology obtained seeing it as
a real algebraic variety of dimension twice the dimension of $Sec _t(X(\Cc) )$.

\begin{theorem}\label{i0}
Fix an integer $t>0$ and assume that $Sec _s(X(\Cc))$ is generically identifiable. Then the set of all real points $P\in Sec _s(X(\Cc ))\cap \PP^r(\Rr )$ with
one of the labels $(s,a)$, $0\le a \le \lfloor s/2\rfloor$ is dense in the smooth part $Sec _s(X(\Cc ))\cap \PP^r(\Rr )$ for the euclidean topology and its complementary
is contained in a proper closed subset of $Sec _s(X(\Cc ))\cap \PP^r(\Rr )$ for the Zariski topology.
\end{theorem}

\begin{proof}
We may assume $s\ge 2$, because $Sec _1(X(\Cc )) =X(\Cc )$. 
Since $Sec _s(X(\Cc ))$ is generically identifiable, it has the expected dimension $(s+1)\dim(X)-1$. 
\\ 
Let $E$ be the set of all subsets of $X(\Cc)$ formed by $s$ linearly independent points. 
For each $S\in E$, the
map $S\mapsto \langle S\rangle$ defines a morphism $\phi$ from $E$ to the Grassmannian $G(s-1,r)(\Cc)$ of all
$(s-1)$-dimensional $\Cc$-linear subspaces of $\PP^r (\Cc)$. Let $I:= \{(x,N)\in \PP^r\times G(s-1,r)\; | \; x\in N\}$ be the incidence
correspondence and let $\pi _1: I\to \PP^r(\Cc)$
and $\pi _2: I\to  G(s-1,r)(\Cc)$ denote the restriction to $I$ of the two projections. 
The set $U_{\Cc,a}$  of all points with rank exactly equal to $s$ is the intersection with $\{P\in \PP^r(\Cc) \, | \, rk(P)\geq s\}$ with $\pi _1(\pi _2^{-1}(\phi (E))$ and hence it is constructible. Since $X$ is real and the embedding $X\subset \PP^r$ is defined over $\RR$, we have $\sigma (U_{\Cc ,a}) =U_{\CC ,a}$.
By assumption there is a non-empty open subset $V_{\Cc ,s}$ of $U_{\Cc ,s}$ for the Zariski topology such that each $P\in V_{\Cc ,s}$ comes from a unique point
of $\pi_1( \pi_2^{-1}(E))$ and in particular it is associated to a unique $S_P\in E$. Since the embedding is real, each point of $\sigma (V_{\Cc ,s})$ has the same
property. Since $V_{\Cc ,s}$ is Zariski dense in $Sec _s(X(\Cc))$, $V_{\Cc ,s}\cap \PP^r(\Rr)$ is Zariski dense in $Sec _s(\Cc)\cap \PP^r(\Rr)$.
Fix any $P\in V_{\Cc,s}\cap \PP^r(\Rr)$. Since $S_P$ is uniquely determined by $P$ and $\sigma (P)=P$, we have $\sigma (S_P)=S_P$.
Hence $S_P$ has a label. Let $U'_{\Cc ,s}$ be the subset of $U'_{\Cc ,s}$ corresponding only to sets $S\subset X_{\mathrm{reg}}(\Cc )$
and at which the map $\pi_1$ has rank $s(\dim(X)+1)-1$. Note that $U'_{\Cc ,s}$ is contained in the smooth part of $Sec _s(\Cc)$. Set $V'_{\Cc ,s}:= U'_{\Cc,s}\cap V_{\Cc,s}$.
Since $V'_{\Cc ,s}$ is smooth, $V'_{\Cc,s}\cap \PP^r(\Rr)$ is a smooth real algebraic variety. We saw that each point of $V'_{\Cc,s}$ has a label $(s,a)$. $V'_{\Cc ,s}$ is
dense in the smooth part of the real algebraic variety $Sec _s(X(\Cc))(\Rr ) =Sec _s(X(\Cc))\cap \PP^r(\Rr )$ for the euclidean topology and hence
it is Zariski dense in $Sec _s(X(\Cc))$ and $Sec _s(X(\Cc))\cap \PP^r(\Rr )$.
\end{proof}

\begin{remark} When the set $\Ss (P,X,\Cc )$  of all $S\subset X(\Cc )$ evincing $r_{X(\Cc )}({P})$, is finite, in order to have
$\Ss (P,X,\Cc )(\Rr ) \ne \emptyset$ it is sufficient that $\sharp (\Ss (P,X,\Cc )(\Rr ))$ is odd (we use this observation to prove Theorem \ref{oo1.1} for odd $d$).
\end{remark}

There are many uniqueness results for submaximal tensors (see  for example \cite{bco}, \cite{chio}, \cite{cov}) and so Theorem \ref{i0} applies in many cases.

A first interesting case where to study the admissible rank is the one of rational normal curves. If the degree if the curve is odd, then a general point is identifiable so Theorem \ref{i0} assures the existence of a dense set of points with any label $(1+d/2,a)$,
 while if $d$ is even, the euclidean dimension of points with admissible rank $\lceil (d+1)/2\rceil$ has to be studied. This is the purpose  of Theorem \ref{oo1.1}.
We need to recall the following definition.
\begin{definition}\label{vsp}
Fix any $P\in V$ and let $VSP(P)$ denote
the set of all $S\subset X(\Cc )$ such that $P\in \langle S\rangle _\Cc$ and $\sharp (S) =\rho$. 
\end{definition}

\begin{proof}[Proof of Theorem \ref{oo1.1}:] If $d$ is odd, then a general $P\in \PP^d(\Cc )$ has rank $\lfloor (d+1)/2\rfloor$
and $Sec _{\lfloor (d+1)/2\rfloor}(X(\Cc ))$ is generically identifiable (\cite[Theorem 1.40]{ik}). Thus we may apply Theorem \ref{i0} in this case. 

Now assume
that $d$ is even. In this case a general $q\in \PP^d(\Cc )$ has rank $1 + d/2$ and the set $VSP(q)$ has dimension $1$. Fix $p\in X(\Rr)$.

\medskip 

\quad {\emph {Claim :}} There is a non-empty open subset $\Uu$ of $\PP^d(\Rr)$ such that $\PP^d(\Rr )\setminus \Uu$ has euclidean dimension $<d$ and
each $q\in \Uu$ has admissible rank $1+d/2$ with label $(1+d/2,a)$ and $2a \le d/2$, computed by a set $S$ with $\sigma (S) =S$ and $p\in S$.

\medskip

\quad {\emph{ Proof of the Claim:}} If $d=2$ we can simply take as $\Uu$ the subset of $\PP^2(\Rr)\setminus X(\Rr)$ formed by the points that are not on the tangent
line to $X(\Rr )$ at $p$. 

Therefore assume
$d\ge 4$. Let $\ell : \PP^d(\Cc )\setminus \{p\} \to \PP^{d-1}(\Cc)$ denote the linear projection from $p$. Let $Y(\Cc )$ denote the closure
of $\ell (X(\Cc )\setminus \{p\})$. Since $p\in \PP^d(\Rr)$, $\ell$ is defined over $\Rr$ and $Y(\Cc)$ is defined over $\Rr$. By construction the curve $Y(\Cc )$ is a rational normal curve of degree $d-1$
and $Y(\Cc )\setminus \ell (X(\Cc )\setminus \{p\})$ is a unique point, $p'$, corresponding to the tangent line of $X(\Cc )$ at $p$. Since $p\in X(\Rr)$ and $\ell$ is defined
over $\Rr$, then $p'\in Y(\Rr )$.\\
Now $d-1$ is odd, so $Sec _{\lfloor d/2\rfloor}(Y(\Cc ))$ is generically identifiable and we can use Theorem \ref{i0} to find a non-empty open subset $\Vv \subset \PP^{d-1}(\Rr )$ such that $\PP^{d-1}(\Rr )$ has euclidean dimension
$\le d-2$, each $q\in \Vv$ has admissible rank $d/2$ and $\sharp (VPS(q)) =1$ for all $q\in \Vv$. Since $d/2>1$, there is an open subset
$\Vv'$ of $\Vv$ such that $p' \notin S_q$ for all $q\in \Vv '$, where $S_q$ is the unique element of $VSP(q)$. Note that $\Vv \setminus \Vv'$ has euclidean dimension $\le d-2$
and so $\PP^{d-1}(\Rr )\setminus \Vv '$ has euclidean dimension $\le d-2$. 
\\ Now we have simply to lift up $\Vv '$: consider $\Uu ':= \ell ^{-1}(\Vv ')$, the set $\PP^r(\Cc )\setminus \Uu '$ has clearly euclidean dimension $\le d-1$.
Fix any $a\in \Uu'$, call $b:= \ell (a)$ and take $ \{S_b\} := VPS(b)$. We have $\sigma (S_b)=S_b$ and $p'\notin S_b$. Since $p'\notin S_b$, there
is a unique set $S_a\subset X(\Cc)\setminus \{p\}$ such that $\ell (S_a) =S_b$. Now the set $S$ we are looking for is nothing else than $S:= \{p\}\cup S_a$. In fact since $\sigma (S_b)=S_b$, $\ell$ is defined over $\Rr$ and
$p\in X(\Rr)$, we have $\sigma (S)=S$. Hence each $q\in \Uu'$ has admissible rank $\le 1+d/a$. To get $\Uu$ it is sufficient to intersect $\Uu '$ with the set
of all $q\in \PP^r(\Rr)$ with $\Cc$-rank $1+d/2$.\end{proof}

Landsberg and Teitler gave an upper bound concerning the $X$-rank over $\CC$ (\cite[Proposition 5.1]{lt}). Several examples in \cite{bs} show that
this upper bound is not always true, not even for typical ranks, over $\RR$. The case for labels is easier and we adapt the proof of \cite[Proposition 5.1]{lt} in the following way.

\begin{proposition}\label{i00}
Let $X$ be a geometrically integral variety defined over $\Rr$ and equipped with an embedding $X\subset \PP^r$ defined over $\Rr$ and of dimension $m$. 
\begin{itemize}
\item If either $r-m+1$ is even or $X(\Rr)$ is Zariski dense in $X(\Cc)$, then each $P\in \PP^r(\Rr)$ has a label $(s,a)$ with $s\le r-m+1$.
\item If $r-m+1$ is odd and $X(\Rr )$ is not Zariski dense in $X(\Cc)$, then $P$ has a label $(s,a)$ with either $s\le r-m+1$ or $s=r-m+1$ and $a=0$.
\end{itemize}
\end{proposition}

\begin{proof}
Fix $P\in \PP^r(\Rr )$. If 
$P\in X(\Rr )$  it has $(1,0)$ as its unique label. 
\\
Now assume $P\notin X(\Rr )$. 
Let $U$ be the set of all linear spaces $H\subset \PP^r(\Cc)$ defined over $\Cc$, containing $P$ and transversal to $X(\Cc)$. By Bertini's theorem, $U$ is a non-empty open subset of the Grassmannian $\mathbb{G}_{\Cc } := \mathbb{G}(r-m-1,r-1)$ of all $(r-m)$-dimensional linear subspaces of $\PP^r(\Cc )$ containing $P$. Since $X$ and $P$ are defined over
$\Rr$, $\mathbb{G}_{\Cc}$ is also defined over $\Rr$. 
\\
By definition of $U$, for  each $H\in U$ we have that $P\in H$ and  $H$ intersects  $X(\Cc)$ in $\deg(X)$ distinct points. There is a non-empty open subset $V$ of $U$
such that every $H\in V$ has the following property: any $S_1\subseteq H\cap X(\Cc )$ with $\sharp (S_1) = r-m+1$ spans $H$. Since $\mathbb{G}_{\Cc}$ is a Grassmannian, this implies that $V(\Rr )$ is Zariski dense in $V$ and in particular $V(\Rr )\ne \emptyset$.
Take $H\in V(\Rr)$. The set $S:= H\cap X(\Cc )$ is formed by $\deg(X)$ points of $X(\Cc)$ and $\sigma (S) =S$. 
\\
Assume for the moment that
either $r-m+1$ is even or $S\cap X(\Rr ) \ne \emptyset$. We may find $S_1\subseteq S$ such that $\sharp (S_1)=r-m+1$
and $\sigma (S_1) =S_1$. Hence $S_1$ is a set with a label $(r-m+1,a)$ for some integer $a$. 
\\
If $r-m+1$ is odd and $H\cap X(\Rr )=\emptyset$, then we have
a set $S_2\subseteq S$ with $\sharp (S_2)=r-m+2$ with $\sigma (S_2)=S_2$ and hence $P$ has a label $(r-m+2,(r-m+2)/2)$.
\\
Finally if $r-m+1$ is odd and  $X(\Rr )$ is Zariski dense in $X(\Rr )$,  the set of all $H\in V(\Rr )$ with $H\cap X(\Rr )\ne \emptyset$ is non-empty
(and Zariski dense in $V$). \end{proof}

\subsection{Typical $a$-ranks}\label{S3}
As above, let $X\subset \PP^r$ be a geometrically integral non-degenerate variety defined over $\Rr$ and with a smooth real point, i.e. with
$X(\Rr)$ Zariski dense in $X(\Cc)$. 
\begin{notation}
We use $\langle \ \rangle $ or $\langle \ \rangle _\Cc$ to denote linear span over $\Cc$ inside $\PP^r(\Cc)$.
For any $S\subseteq \PP^r(\Rr )$ let $\langle S\rangle _\Rr$ be its linear span in $\PP^r(\Rr)$. Note that $\langle S\rangle _\Cc \cap \PP^r(\Rr ) =\langle S\rangle _\Rr$.
\end{notation}
Here we consider the case in which $Sec _s(X) =\PP^r$ and hence almost never we have generic uniqueness and we won't be able to apply Theorem \ref{i0}. 
The notion of typical rank may be generalized in the following way. 

Fix an integer $a\ge 0$. For any set of complex points $P_1,\dots ,P_a\in X(\Cc)\setminus X(\Rr)$
the linear space $L:= \langle \{P_1,\sigma (P_1),\dots ,P_a,\sigma (P_a)\}\rangle _\Cc$ is defined over $\Rr$ and hence $L\cap \PP^r (\Rr)$ is a linear space
with $\dim _\Rr L\cap \PP^r(\Rr) = \dim _\Cc L$ and $L = (L\cap \PP^r (\Rr))_\Cc$.

\begin{definition}[$a$-$X(\Rr)$-rank]\label{aXtyp}
 For any $P\in \PP^r(\Rr)$, the \emph{$a$-$X(\Rr)$-rank} of $P$ is the minimal integer $c$
such that there are $P_1,\dots ,P_a\in X(\Cc)\setminus X(\Rr)$ and $Q_1,\dots ,Q_c\in X(\Rr)$ such that $P \in \langle \{P_1,\sigma (P_1),\dots ,P_a,\sigma (P_a),Q_1,\dots ,Q_c\}\rangle _\Cc$. 
\end{definition}
The $a$-typical $X(\Rr )$-ranks are the integers occurring as $a$-ranks a non-empty euclidean open subset of $\PP^r(\Rr )$. 

Note that $0$ is typical
if and only if $Sec _a(X(\Cc )) =\PP^r(\Cc )$.
The proof of \cite[Theorem 1.1]{bbo} easily prove the following result.
\begin{proposition}\label{oo1}
All the integers between two different $a$-typical ranks are $a$-typical ranks.
\end{proposition}
One may wonder if a label for a general element in $\PP^r (\Rr)$ of certain typical rank always exists. The answer is ``~no~" and we present an example in the following section by using varieties of sum of powers. 

\subsection{Variety of Sum of Powers} 

We indicate with $\rho$ the generic $X$-rank of $\PP^r(\Cc)$
and with $\Uu \subset \PP^r(\Cc )$  a non-empty open subset
of $\PP^r(\Cc )$ of points of generic rank: $r_{X(\Cc )}(P) = \rho$ for all $P\in \Uu$. Since $X$ is defined over $\Rr$,  also all point $P\in \sigma (\Uu)$ have the same generic rank.
Now if $$U:= \Uu \cap \sigma (\Uu)
\hbox{ and }V:= U\cap \PP^r(\Rr),$$  then the real part $V$  of $U$ is a non-empty open subset of $\PP^r (\Rr)$ whose complement has dimension smaller than $\rho$. 
For all the points $P$ of generic rank, the variety of sum of powers $VSP(P)$, recalled in  Definition \ref{vsp}, is non-empty. Since
$P\in \PP^r$, we have $\sigma (VSP(P)) = VSP(P)$ and hence the constructible set $VSP(P)$ is defined over $\Rr$. 
Now the real part of  $VSP(P)$,  $VSP(P)(\Rr )$ is non empty if and
only if there is an integer $a$ with $0\le a\le \rho /2$ and $P$ has a label $(\rho ,a)$. In a few cases $VSP(P)$ is known (see e.g. \cite{bc}, \cite{bco}, \cite{cov}).

The following example says that $VSP(P)$ may not have a label for a general $P\in \PP^r (\Rr)$ with generic rank over $\Cc$. However in this case $VSP(P)$ is finite
and we do not have positive dimensional examples. If $VSP(P)$ is finite, we have $VSP(P)(\Rr )\ne \emptyset$ (and so $P$ has a label) if $\sharp (VSP(P))$ is odd.

\begin{example}\label{bo1}
Let $C$ be a smooth elliptic curve defined over $\Rr$. Take two points $P'_1,P'_2\in C(\Cc )\setminus C(\Rr )$ and set
$P''_i:= \sigma (P'_i)$. By assumption $P''_i\ne P'_i$. For general $P'_1,P'_2$, we have $P''_1\ne P'_2$ and so $P'_1 \ne P'' _2$.

Let $E$ be a geometrically integral curve of arithmetic genus $3$ and with exactly $2$ singular points $O'$ and $O''$, both of them ordinary nodes, where $C$
has
its normalization and with $O'$ obtained gluing together the set $\{P'_1,P'_2\}$ and $O''$ obtained gluing together the set $\{P''_1,P''_2\}$. The involution $\sigma$
acts on $E$ with $\sigma (O')=O''$. For general $P'_1,P'_2$ the divisors $P'_1+P'_2$ and $P''_1+P''_2$ are not linearly equivalent
and hence $E$ is not hyperelliptic. Hence $\omega _E$ is very ample and embeds $E$ in $\PP^2(\Cc )$ as a degree $4$ curve with $2$ ordinary nodes, $Q_1$, $Q_2$,
and no other singularity.  Since $\sigma (O') =O''$ and $O'\ne O''$, we have that for $i=1,2$, $\sigma (Q_i) =Q_{3-i}$ and $Q_i\in \PP^2(\Cc)\setminus \PP ^2(\Rr)$.

Let $u: C \to E$ denote the normalization map. Set $\Oo _C(1):= u^\ast (\omega _E)$. Since $\omega _C\cong \Oo _C$, we have
$P'_1+P''_1+P'_2+P''_2 \in |\Oo _C(1)|$. Thus $\Oo _C(1)$ is a degree $4$ line bundle on $C$ defined over $\Rr$ and the complete linear system $|\Oo _C(1)|$ induces an embedding $j: C\to \PP^3(\Cc)$ defined over
$\Rr$ and with $X:= j({C})$ a smooth elliptic curve of degree $4$. The pull-back by $u$ of the linear system $|\omega _E|$ is a codimension $1$ linear subspace of the $3$-dimensional
$\Cc$-linear space $|\Oo _C(1)|$, the one associated to the $\Cc$-vector space of all rational $1$-forms on $C$ with poles only at the points $P'_1$, $P'_2$, $P''_1$, $P''_2$,
each of them at most of order one and such that the sum of their residues at the $4$ points $P'_1$, $P'_2$, $P''_1$, $P''_2$ is zero.
Since $\sigma (\{P'_1,P'_2,P''_1,P''_2\}) =\{P'_1,P'_2,P''_1,P''_2\}$, this $2$-dimensional linear subspace of $|\Oo _C(1)|$ is defined by one equation $\ell =0$ with $\sigma (\ell )=\ell$, i.e.
it is defined over $\Rr$. Since this is the linear system $|\omega _E|$ whose image is the curve that we indicate with $D$, then $D$
is defined over $\Rr$.
Since $D$ is defined over $\Rr$, the curve $D$ is the image of $X(\Cc)$ by the linear projection from a point $P\in \PP^3(\Rr )\setminus X(\Rr )$. By construction  $X$ has exactly $2$ secant lines passing through $Q$ and they are the lines $L'$, $L''$ with $L'$ spanned by $j(P'_1)$ and $j(P'_2)$ and $L''$ spanned by $j(P''_1)$ and $j(P''_2)$. Now, since $L'\ne L''$, we get
that $Q$ has no label. Now take $P\in \PP^3(\Rr )$ near $Q$ in the euclidean topology. The linear projection from $Q$ is a geometrically integral curve $D_p$ defined
over $\Rr$, with $C$ as its normalization and near $D$ and so with two ordinary nodes $Q_1({P})$ and $Q_2({P})$ near $Q_1$ and $Q_2$. Thus none of these
nodes is defined over $\Rr$. Therefore there is an euclidean neighborhood $V$ of $Q$ in $\PP^3(\Rr)$ such that no $P\in V$ has a label.\end{example}

\begin{remark}
 If $VSP(P)$ is finite, we have $VSP(P)(\Rr )\ne \emptyset$ (and so $P$ has a label) if $\sharp (VSP(P))$ is odd. If $r=3$ and $X$ is a smooth curve of genus $g$ and degree
 $d$ the genus formula for plane curves gives $\sharp (VSP(P)) = (d-1)(d-2)/2 -g$. When $\Ss (P,X,\Cc )$ is infinite, it is not clear that at least one irreducible component
of $\Ss (P,X,\Cc )$ is $\sigma$-invariant and hence defined over $\Rr$.
\end{remark}

\subsection{Polynomials with admissible rank bigger than complex rank}\label{ex2}

We use elliptic curves and \cite[Example 3.4]{bc} to construct an example of a pair $(X,P)$ with $\sharp (\Ss (P,X,\Cc ))=2$ and $\Ss (P,X,\Cc )(\Rr ) =\emptyset$ with $P$ a symmetric tensor and with $X$ a $d$-Veronese embedding
of $\PP ^n$, $n\ge 2$ and $d$ even (here $r = \binom{n+d}{n} -1$) and $r_{X(\Cc )}({P}) = 3d/2$.

\begin{example}\label{iii1} We consider the Veronese variety $X$ of dimension $n\ge 2$  embedded into $\PP^{\binom{n+d}{n}-1}$ with the complete linear system $|\mathcal{O}(d)|$ with 
 $d\ge 6$ and $d$ even. Set $k:= d/2$. 
The projective space $\PP^n$ and the embedding 
are defined over $\Rr$.  

Let $E$ be a smooth curve of genus $1$ defined over $\Rr$ and with $E(\Rr )\ne \emptyset$ (so if $\Rr =\RR$, then $E(\Rr)$ is homeomorphic either to a circle or to the disjoint union of $2$ circles). Fix $O\in E(\Rr )$. By Riemann-Roch the complete linear system $|\Oo _E(3O)|$ defines an embedding $j: E\to \PP^2$ defined over $\Rr$ and with
$j(E)$ a smooth plane cubic. With an $\Rr$-linear change of homogeneous coordinates $x,y,z$ we may assume that $j(O) =(0:1:0)$ and that $j(E) = \{y^2z= x^3+Axz^2+Bz^3\}$
for some $A, B\in \Rr$. The restriction to $j(E)\setminus \{O\}$ of the linear projection $\PP^2\setminus \{O\}\to \PP^1$ induces the morphism $\phi : E\to \PP^1$ which is induced by the complete linear system $|\Oo _E(2O)|$. Since $O\in E(\Rr )$, $\phi$ is defined over $\Rr$. If $t\in \Rr$ and $-t \gg 0$, then $t^3+At+B<0$ and so $y^2=t^3+At+B$ has two solutions $q', q'' \in \Cc$ with $q'' =\sigma (q')$. Hence we may find $S' =\{q_1,\dots ,q_{3k}\}\subset E(\Cc )\setminus E(\Rr)$ such that $\sharp (S'\cup \sigma (S')) =6k$ and $q_i+\sigma (q_i)
\in |\Oo _E(2O)|$ for all $i$. Note that $\sum _{i=1}^{3k} q_i + \sum _{i=1}^{3k} \sigma (q_i) \in |\Oo _E(6kO)|=|\mathcal{O}_E(3dO)|$.

Now fix a plane $M\subseteq \PP^n$ defined over $\Rr$ and identify the $\PP^2$ where $j(E)$ is embedded  with $M$ so that $j(E)\subset \PP^n$. If $\nu_d$ is the Veronese embedding that maps $\mathbb{P}^n$ to $X$, we can now embedd $E$ into $X\subset \mathbb{P}^{r:= \binom{n+d}{n}-1}$ and define $C:= \nu _d(j(E))$, $Q:= \nu _d(j(O))$, $S:= \nu _d(S')$ and consider the space spanned by $C$: $\Lambda _\Cc := \langle C\rangle _\Cc$. Since the smooth plane curve $j(E)$ is projectively normal,
the embedding $E\to \Lambda _\Cc$ is induced by the complete linear system $|\Oo _C(3dQ)|$ and so $\dim _\Cc \Lambda _\Cc = 3d-1 =6k-1$. Since $\nu _d$, the inclusion $M
\subseteq \PP^n$ and $j$ are defined over $\Rr$, $\Lambda _\Cc \cap \PP^r(\Rr )$ is an $\Rr$-projective space of dimension $6k-1$.
Let $Z\subset C(\Cc)$ be any set with $a:= \sharp (Z) \le 6k$. Since $C$ is embedded in $\Lambda _\Cc$ by the complete linear system $|\Oo _C(6kQ)|$, Riemann-Roch gives
that $\dim _\Cc (\langle Z\rangle _\Cc )=a-1$, unless $a=6k$ and $Z\in |\Oo _C(6kQ)|$. If $Z\in |\Oo _C(6kQ)|$, then $\langle Z\rangle _\Cc$ is a hyperplane of $\Lambda _\Cc$.
Hence $\dim _\Cc (\langle S\rangle )=\dim (\langle \sigma (S)\rangle ) =3k-1$ and $\dim _\Cc (\langle S\cup \sigma (S) \rangle )= 3k-2$.
By the Grassmann's formula the set $\langle S\rangle \cap \langle \sigma (S)\rangle$ is a unique point, $P$. Since $\sigma ({P}) \in \langle S\rangle \cap \langle \sigma (S)\rangle$,
we have $\sigma ({P}) =P$, i.e. $P\in \PP^r(\Rr )$. Since $P\in \langle S\rangle _\Cc$ and $C\subset X$, we have $r_{X(\Cc )}({P}) \le 3k$.
Over $\Cc$ this is the construction of \cite[Example 3.2]{bc}. In \cite[Example 3.2]{bc} it is proved that $r_{X(\Cc )}({P}) =3k$ and that $\sharp (\Ss (P,X,\Cc ))=2$. Hence $\Ss (P,X,\Cc ) =\{S,\sigma (S)\}$.
Since $S\ne \sigma (S)$, we have $\Ss (P,X,\Cc )(\Rr)=\emptyset$ and hence $P$ has admissible rank $>3k$.
\end{example}

\providecommand{\bysame}{\leavevmode\hbox to3em{\hrulefill}\thinspace}

\end{document}